\documentclass[preprint, authoryear]{elsarticle}
\usepackage[utf8]{inputenc}

\usepackage{amsmath}
\usepackage{amssymb}
\usepackage{amsthm}
\usepackage{amsfonts}
\usepackage{graphicx}
\usepackage{natbib}
\usepackage[table]{xcolor}

\newtheorem{theorem}{Theorem}
\newtheorem{lemma}{Lemma}
\newtheorem{corollary}{Corollary}
\newtheorem{assumption}{Assumption}

\title{Testing for subsphericity when $n$ and $p$ are of different asymptotic order\tnoteref{t1}}
\author[1]{Joni Virta}
\ead{joni.virta@utu.fi}

\address[1]{Department of Mathematics and Statistics, University of Turku, Finland}

\fntext[fn1]{Funding: This work was supported by the Academy of Finland [335077].}

\makeatletter
\def\ps@pprintTitle{%
 \let\@oddhead\@empty
 \let\@evenhead\@empty
 \def\@oddfoot{}%
 \let\@evenfoot\@oddfoot}
\makeatother

\begin{document}

\begin{abstract}
We extend a classical test of subsphericity, based on the first two moments of the eigenvalues of the sample covariance matrix, to the high-dimensional regime where the signal eigenvalues of the covariance matrix diverge to infinity and either $p/n \rightarrow 0$ or $p/n \rightarrow \infty$. In the latter case we further require that the divergence of the eigenvalues is suitably fast in a specific sense. Our work can be seen to complement that of Schott (2006) who established equivalent results in the case $p/n \rightarrow \gamma \in (0, \infty)$. As our second main contribution, we use the test to derive a consistent estimator for the latent dimension of the model. Simulations and a real data example are used to demonstrate the results, providing also evidence that the test might be further extendable to a wider asymptotic regime. 
\end{abstract}


\maketitle

\section{Introduction}\label{sec:intro}

The objective of principal component analysis (PCA), and dimension reduction in general, is to extract a low-dimensional signal from noise-corrupted observed data. The most basic statistical model for the problem is as follows. Assume that $S_n$ is the sample covariance matrix of a random sample from a $p$-variate normal distribution whose covariance matrix has the eigenvalues $\lambda_1 \geq \cdots \geq \lambda_d > \sigma^2, \ldots, \sigma^2$ exhibiting ``spiked'' structure. The data can thus be seen to be generated by contaminating a random sample residing in a $d$-dimensional subspace with independent normal noise having the covariance matrix $\sigma^2 I_p$. This signal subspace can be straightforwardly estimated with PCA as long as one knows its dimension $d$ which is, however, usually unknown in practice. Numerous procedures for determining the dimension have been proposed, see \cite{jolliffe2002principal} for a review and, e.g.,  \cite{schott2006high,nordhausen2016asymptotic,virta2019estimating} for asymptotic tests and  \cite{beran1985bootstrap,dray2008number,luo2016combining} for bootstrap- and permutation-based techniques. Simplest of these methods is perhaps the test of sub-sphericity based on the test statistics,
\begin{align*}
    T_{n,j} = \frac{m_{2,p-j}(S_n)}{m_{1,p-j}(S_n)^2} - 1, \quad j = 0, \ldots, p - 1,
\end{align*}
where $m_{\ell,r}(A)$ denotes the $\ell$th sample moment of the last $r$ eigenvalues of the symmetric matrix $A$. Under the null hypothesis $H_{0k}: d = k$ that the signal dimension equals $k$, the limiting null distribution of $T_{n,k}$ is
\begin{align}\label{eq:low_dim_convergence}
    \frac{1}{2} n (p - k) T_{n,k} \rightsquigarrow \chi^2_{\frac{1}{2}(p - k)(p - k + 1) - 1},
\end{align}
as $n \rightarrow \infty$, see, e.g., \cite{schott2006high}. Hence, the dimension $d$ can in practice be determined by testing the sequence of null hypotheses $H_{00}, H_{01}, \ldots$ and taking the estimate of $d$ to be the smallest $k$ for which $H_{0k}$ is not rejected. By examining the power of the tests, \cite{nordhausen2016asymptotic} concluded that this procedure yields a consistent estimate of $d$ (with a suitable choice of test levels).

The previous test assumes a fixed dimension $p$ and, in the face of modern large and noisy data sets with great room for dimension reduction, it is desirable to extend the test to the high-dimensional regime where $p = p_n$ is a function of $n$ and we have $p_n \rightarrow \infty$ as $n \rightarrow \infty$. This is discussed in Section \ref{sec:theory} where our first main contribution, extending the test based on  \eqref{eq:low_dim_convergence} to the high-dimensional regime where either the sample size or the dimension asymptotically dominates the other, is also presented. Section \ref{sec:power} introduces our second main contribution, a power study of the test, using which we construct a consistent estimator for the true latent dimension. In Section \ref{sec:simulation} we demonstrate our results using simulations and a real data example and, in Section \ref{sec:discussion}, we finally conclude with some discussion.

\section{High-dimensional testing of subsphericity}\label{sec:theory}

The behaviour of most high-dimensional statistical procedures depends crucially on the interplay between $n$ and $p_n$ and the most common approach in the literature is to assume that their growth rates are proportional in the sense that $p_n/n \rightarrow \gamma \in (0, \infty)$ as $n \rightarrow \infty$, see, e.g., \cite{yao2015sample}. The limiting ratio $\gamma$ is also known as the \textit{concentration} of the regime. In \cite{schott2006high}, the test of subsphericity discussed in Section \ref{sec:intro} is extended to this asymptotic regime under the following two assumptions (note that in Assumption \ref{assu:eigenvalues} the signal dimension $d$ is a constant not depending on $n$).

\begin{assumption}\label{assu:distribution}
The observations $x_1, \ldots , x_n$ are a random sample from $\mathcal{N}_{p_n}(\mu_n, \Sigma_n)$ for some $\mu_n \in \mathbb{R}^{p_n}$ and some positive-definite $\Sigma_n \in \mathbb{R}^{p_n \times p_n}$. 
\end{assumption}

\begin{assumption}\label{assu:eigenvalues}
The eigenvalues of the matrix $\Sigma_n$ are $\lambda_{n1} \geq \cdots \geq \lambda_{nd} > \sigma^2 = \cdots = \sigma^2$ for some $\sigma^2 > 0$. Moreover, the eigenvalues $\lambda_{nk}$, $k = 1, \ldots , d$, satisfy $\lambda_{nk} \rightarrow \infty$.
\end{assumption}

In fact, \cite{schott2006high} additionally required that the quantities $\lambda_{nk}/\mathrm{tr}(\Sigma_n)$ converge to positive constants summing to less than unity, but applying our Lemma \ref{lem:wishart_2} in the proof of their Theorem 4 reveals that this condition is unnecessary, see \ref{sec:schott} for details. Hence, denoting by $S_n$ the sample covariance matrix of the observations, under Assumptions \ref{assu:distribution} and \ref{assu:eigenvalues} and $\gamma \in (0, \infty)\setminus\{ 1\} $ (see \ref{sec:schott} for more details on the exclusion of the case $\gamma = 1$), Theorem~4 in \cite{schott2006high} establishes that the test statistic,
\begin{align*}
    T_{n, j} :=  \frac{m_{2,p_n - j}(S_n)}{m_{1,p_n - j}(S_n)^2} - 1,
\end{align*}
satisfies $ (n - d - 1) T_{n, d} - (p_n - d) \rightsquigarrow \mathcal{N}(1, 4)$ where $d$ is the signal dimension. 
As remarked by \cite{schott2006high}, this limiting result is consistent with its low-dimensional equivalent \eqref{eq:low_dim_convergence} in the sense that, as $p \rightarrow \infty$,
\begin{align*}
    \frac{2}{p - d}\chi^2_{\frac{1}{2}(p - d)(p - d + 1) - 1} - (p - d) \rightsquigarrow \mathcal{N}(1, 4).
\end{align*}

A crucial condition that allows the above limiting result is the divergence of the spike eigenvalues $\lambda_{n1}, \ldots , \lambda_{nd}$ of the covariance matrix to infinity in Assumption \ref{assu:eigenvalues}. Indeed, usually the spikes are taken to be constant in the literature for high-dimensional PCA, see, e.g. \cite{baik2006eigenvalues,johnstone2018pca}. However, requiring the spikes to diverge to infinity is rather natural and reflects the idea that only a few principal components are sufficient to recover a large proportion of the total variance even in high dimensions. See, for example, \cite{yata2018test}, who use cross-data-matrices to detect spiked principal components with divergent variance, and the references therein. 


As our first contribution, we extend the result of \cite{schott2006high} outside of the regime $p_n/n \rightarrow \gamma \in (0, \infty)$, to the extreme cases $\gamma \in \{ 0 , \infty \} $. The latter have been less studied in the high-dimensional literature, but see, for example, \cite{karoui2003largest,birke2005note,yata2009pca,jung2009pca}, the last of which consider the extreme asymptotic scenario where the dimension diverges to infinity but the sample size remains fixed. In our treatment of the case $\gamma = \infty $, we further require the additional condition that $p_n/(n \sqrt{\lambda_{nd}}) \rightarrow 0$ as $n \rightarrow \infty$, i.e., the dimension must not diverge too fast compared to the sample size and the magnitude of the spike $\lambda_{nd}$ corresponding to the weakest signal. Assumptions of this form are rather common in high-dimensional PCA when the spikes are taken to diverge, see, e.g., \cite{shen2016general} who saw $n$, $\lambda_{nk}$ and $p_n$ as three competing forces affecting the consistency properties of PCA, $n$ and $\lambda_{nk}$ contributing information about the signals and $p_n$ decreasing the relative share of information in the sample by introducing more noise to the model. The condition $p_n/(n \sqrt{\lambda_{nd}}) \rightarrow 0$ can thus be interpreted as requiring that even the weakest of the spike principal components has asymptotically strong enough signal to be detected.

The extension of the test to the previous regimes is given below in Theorem~\ref{theo:goes_to_zero}. The main line of proof is based on extending the work of \cite{birke2005note}, who considered testing of sphericity in the cases $\gamma \in \{ 0 , \infty \}$, to testing of subsphericity. In this sense, our work is to \cite{birke2005note} what \cite{schott2006high} is to \cite{ledoit2002some}, who studied tests of sphericity in the case  where $\gamma \in (0, \infty) $ and on whose work \cite{schott2006high} based their proof.

\begin{theorem}\label{theo:goes_to_zero}
Under Assumptions \ref{assu:distribution} and \ref{assu:eigenvalues}, if, as $n \rightarrow \infty$, either
\begin{enumerate}
    \item[i)] $p_n/n \rightarrow 0$, or,
    
    \item[ii)] $p_n/n \rightarrow \infty $ and $p_n/(n \sqrt{\lambda_{nd}}) \rightarrow 0$, then,
\end{enumerate}
\begin{align*}
    (n - d - 1) T_{n, d} - (p_n - d) \rightsquigarrow \mathcal{N}(1, 4). 
\end{align*}
\end{theorem}

\section{Power analysis and dimension estimation}\label{sec:power}

A natural question is whether the test of subsphericity can be used to consistently estimate the latent dimension $d$ under the high-dimensional Gaussian model. In a low-dimensional setting, this is accomplished by chaining together tests for $H_{0k}: d = k$ for different values of $k$ in some specific order: In forward testing one sequentially tests for $H_{00}, H_{01}, \ldots$ and takes as the estimate of $d$ the smallest $k$ for which $H_{0k}$ is not rejected. In backward testing, the order is $H_{0(p-1)}, H_{0(p-2)}, \cdots$ and the estimate is the largest $k$ for which $H_{0(k-1)}$ is rejected. The two strategies can also be combined into a ``divide-and-conquer'' approach where one starts from the middle of the search interval and subsequently halves it with each test, this process often terminating in fewer tests than the forward and backward testing. However, in the high-dimensional setting where our working assumption is that the number of latent signals is diminutive compared to the overall dimensionality (finite $d$ vs. $p_n \rightarrow \infty$), the most economic choice is likely the forward testing. In the following we show that this strategy indeed leads, under suitable assumptions, to a consistent estimate of the dimension $d$ in various high-dimensional regimes. Even though the equivalent of Theorem \ref{theo:goes_to_zero} for $ \gamma \in (0, \infty)\setminus\{ 1\} $ was established already in \cite{schott2006high}, the following results are novel also in that case. We use the notation $g_{n,k} := (n - k - 1) T_{n, k} - (p_n - k)$, $k = 0, \ldots , p_n - 1$, for the test statistic.

\begin{theorem}\label{theo:power_for_small_k}
Under Assumptions \ref{assu:distribution} and \ref{assu:eigenvalues}, if, as $n \rightarrow \infty$, either
\begin{enumerate}
    \item[i)] $p_n/n \rightarrow \gamma \in [0, \infty) \setminus \{ 1 \}$ and $p_n/\lambda_{nd}^2 \rightarrow 0$, or,
    \item[ii)] $p_n/n \rightarrow \infty $, $p_n/(n \sqrt{\lambda_{nd}}) \rightarrow 0$ and $p_n/(\sqrt{n} \lambda_{nd}) \rightarrow 0$, then,
\end{enumerate}
we have, for each $k = 0, \ldots , d - 1$ and for all $ M > 0 $, that
\begin{align*}
    \mathbb{P}( g_{n,k}/n \leq M ) \rightarrow 0.
\end{align*}
\end{theorem}

Theorem \ref{theo:power_for_small_k} shows that the test for $H_{0k}$ is consistent under the alternative hypothesis that the true dimension $d > k$ (the power of the test in the opposite case $d < k$ plays no role in the forward testing and, hence, is not studied here). As a straightforward corollary we then obtain the consistency of the forward testing.

\begin{corollary}\label{cor:dimension_estimation}
Under the assumptions of Theorem \ref{theo:power_for_small_k}, let $c_n$ be any sequence of real numbers satisfying $c_n \rightarrow \infty$ and $c_n = \mathcal{O}(n)$ as $n \rightarrow \infty$. Then,
\begin{align*}
    \hat{d} := \min \{ k = 0, \ldots , p_n - 1: g_{n, k} \leq c_n \} \rightarrow_p d.
\end{align*}
\end{corollary}

Choosing a sequence $c_n$ for which the forward testing estimator $\hat{d}$ performs well in finite samples is a highly non-trivial task and, thus, we advocate using in practice the alternative estimator,
\begin{align}\label{eq:practical_estimator}
    \hat{d} := \min \{ k = 0, \ldots , p_n - 1: | (g_{n, k} - 1)/2 | \leq z_{1-\alpha/2} \},
\end{align}
where $z_{1-\alpha/2}$ is the upper $\alpha/2$ quantile of the standard normal distribution, see, e.g., \cite{nordhausen2016asymptotic} for a similar modification. The resulting procedure has asymptotically zero probability to underestimate the dimension (by Theorem \ref{theo:goes_to_zero}) and carries the Type I error probability equal to $\alpha$ of overestimating the dimension (by Theorem \ref{theo:power_for_small_k}). 

Finally, we still briefly discuss the assumptions of Corollary \ref{cor:dimension_estimation} which, while stricter than in Theorem \ref{theo:goes_to_zero}, can nevertheless be seen to be very natural. That is, regardless of the regime, the assumptions ask that the weakest of the signals is strong enough not to be masked by the noise (similarly as in part \textit{ii)} of Theorem~\ref{theo:goes_to_zero}). To gain a more concrete idea on the severity of the assumptions, let $p_n = c n^\alpha$ and $\lambda_{nd} = n^\beta$ for some $c \neq 1$ and $\alpha, \beta > 0$. Then, the feasible values of $(\alpha, \beta)$ form a polygon in $\mathbb{R}^2$ that is illustrated in the range $0 < \alpha \leq 2$ as the grey area in Figure~\ref{fig:rev_plot_1}. The plot reveals the intuitive fact that the effect of the dimension on the minimal feasible growth rate for the signal is the stronger the faster the dimension increases (the slope of the curve is for $\alpha > 1.5$ four times higher than for $\alpha \in (0, 1)$).

\begin{figure}
    \centering
    \includegraphics[width = 0.75\textwidth]{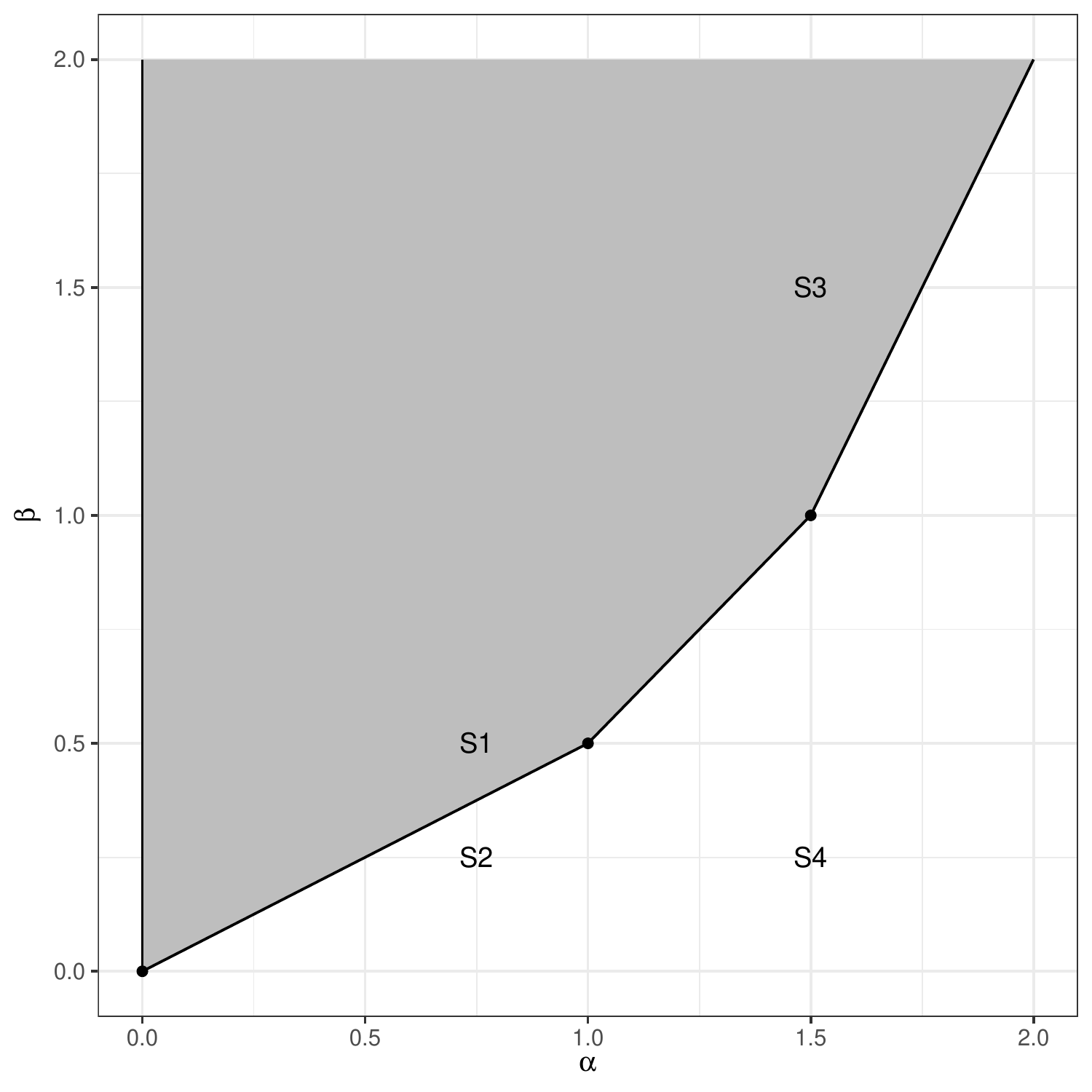}
    \caption{Assuming $p_n = c n^\alpha$ and $\lambda_{nd} = n^\beta$ for some $c \neq 1$ and $\alpha, \beta > 0$, the grey area in the plot contains the values of $(\alpha, \beta)$ for which the assumptions of Corollary \ref{cor:dimension_estimation} hold. The points S1--S4 correspond to the four settings used in the simulation study in Section \ref{sec:simulation}.} 
    \label{fig:rev_plot_1}
\end{figure}

\section{Numerical examples}\label{sec:simulation}

We first demonstrate the result of Theorem \ref{theo:goes_to_zero} using simulated data. We consider four different settings, each of which assumes a sample of size $n$ from $\mathcal{N}_{p_n}(0, \Sigma_{n})$ where $\Sigma_n = \mathrm{diag}(\lambda_{n1}, \ldots , \lambda_{nd}, 1, \ldots , 1)$. Note that this simplified form of the normal distribution (zero location, unit noise variance and diagonal covariance) is without loss of generality as our test statistic is location, scale and rotation invariant. The settings are as follows:
\begin{enumerate}
    \item $d = 3$, $n = 21$6, $p_n = n^{3/4}$, $\lambda_{n1} = 3 n$, and $\lambda_{n2} = \lambda_{n3} = n^{1/2}$,
    \item $d = 3$, $n = 216$, $p_n = n^{3/4}$, $\lambda_{n1} = 3 n^{1/2}$, and $\lambda_{n2} = \lambda_{n3} = n^{1/4}$,
    \item $d = 2$, $n = 36$, $p_n = n^{3/2}$, $\lambda_{n1} = 2 n^2$ and $\lambda_{n2} = n^{3/2}$,
    \item $d = 2$, $n = 36$, $p_n = n^{3/2}$, $\lambda_{n1} = 2 n^2$ and $\lambda_{n2} = n^{1/4}$.
\end{enumerate}
Settings 1 and 2 fall within the case $\gamma = 0$, and their only difference is in the growth rates of the spikes. Settings 3 and 4 explore the case $\gamma = \infty$, the former satisfying the conditions of Theorem \ref{theo:goes_to_zero} and the latter not (again the only difference between them is in the growth rates of the spikes). In each case, we compute 10000 replicates of the test statistic $g_{n,d} = (n - d - 1) T_{n, d} - (p_n - d)$ and plot the obtained histogram superimposed with the density of the limiting distribution $\mathcal{N}(1, 4)$.

\begin{figure}
    \centering
    \includegraphics[width = 1\textwidth]{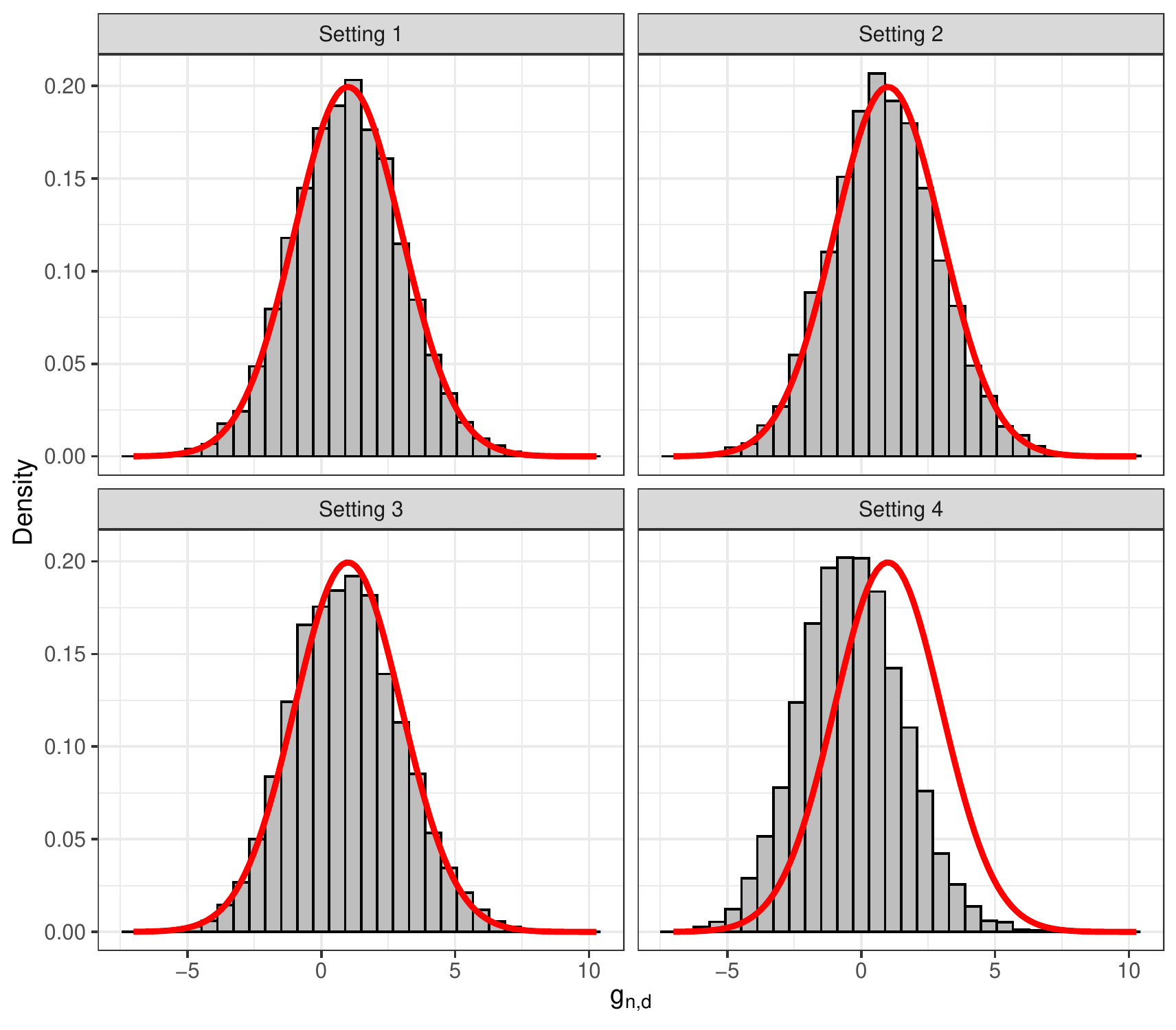}
    \caption{The histograms of 10000 independent replicates of the test statistic $g_{n, d} = (n - d - 1) T_{n, d} - (p_n - d)$ under the four different settings, with the density of the limiting distribution $\mathcal{N}(1, 4)$ overlaid.} 
    \label{fig:rev_plot_2}
\end{figure}

The results are shown in Figure \ref{fig:rev_plot_2} where we immediately make three observations: the convergence to the limiting distribution is (at least visually) rather fast in Settings 1--3, with the histograms exhibiting the Gaussian shape and being only slightly shifted to the left from their limiting density; Setting 1 does not appear to be significantly closer to Gaussianity than Setting 2 despite the increased amount of information in the former (in the form of more rapidly growing spike eigenvalues); in Setting 4 where the condition $p_n/(n \sqrt{\lambda_{nd}}) \rightarrow 0$ required by Theorem \ref{theo:goes_to_zero} is being violated, the histogram visibly has the correct shape and scale, but clearly underestimates the location. The difference between the true mean and the mean of the replicates in Setting 4 is approximately 1.35 and some testing (not shown here) reveals that, at least with the current parameter choices, the difference seems to stay roughly constant when $n$ is increased. Based on this, it seems possible that, even when $p_n/(n \sqrt{\lambda_{nd}}) \nrightarrow 0$, the limiting distribution of $g_{n,d}$ could be made to equal $\mathcal{N}(1, 4)$ with a suitable additive correction term $a_n$, which vanishes, $a_n \rightarrow 0$ as $n \rightarrow \infty$, when the conditions of Theorem \ref{theo:goes_to_zero} are satisfied.

Next, we demonstrate how forward testing, as defined in \eqref{eq:practical_estimator}, can be used to estimate the signal dimension $d$ with a chain of hypothesis tests for the null hypotheses $H_{0k}: d = k$. That is, we sequentially test the null hypotheses $H_{00}, H_{01}, \ldots$ using, respectively, the test statistics $g_{n, 0}, g_{n, 1}, \ldots$ and take our estimate of the dimension to be the smallest $k$ for which $H_{0k}$ is not rejected. For each test, we use $\alpha = 0.05$, i.e., the two-sided $95\%$ critical regions of the limiting $\mathcal{N}(1, 4)$-distribution. We consider the same four settings as in the first simulation, but include an additional, larger sample size for each. Of the four settings, only the first and the third satisfy the assumptions of Corollary \ref{cor:dimension_estimation}, see Figure \ref{fig:rev_plot_1} on how the four settings are located with respect to the ``feasibility region'' of the assumptions.

For simplicity, we report in Table \ref{tab:results_1} the rejection rates (over 10000 replicates) of the null hypotheses corresponding to the true dimension and the neighbouring dimensions only (the columns corresponding to the true dimension are shaded grey). In Settings 1 and 3 where the assumptions of Corollary \ref{cor:dimension_estimation} are satisfied, the test achieves rather accurately the nominal level at the true dimension and shows extremely good power at the smaller dimensions, as expected. Interestingly, the same conclusions are reached also in Setting 2 where the assumptions of Corollary \ref{cor:dimension_estimation} are not satisfied, implying that the assumptions, while sufficient, are not necessary for the consistency of the forward testing estimator. Finally, as expected, the procedure reaches neither a sufficient level nor power in Setting~4 where the conditions of Theorem \ref{theo:goes_to_zero} and Corollary \ref{cor:dimension_estimation} are not satisfied.

\begin{table}[ht]
\caption{The subtables give the observed rejection rates for different null hypotheses over 10000 independent replicates under each of the four settings. Two different sample sizes are considered for each setting. The columns corresponding to the true dimension are shaded grey.} 
\label{tab:results_1}
\centering
\begin{minipage}{.45\linewidth}
\begin{tabular}{p{0.6cm}ccc}
  \multicolumn{4}{c}{Setting 1}\\
  \hline
$n$ & $H_{02}$ & $H_{03}$ & $H_{04}$ \\ 
  \hline
216 & 1.000 & \cellcolor{black!15} 0.053 & 0.115 \\ 
  512 & 1.000 & \cellcolor{black!15} 0.051 & 0.138 \\ 
   \hline
\end{tabular}
\end{minipage}
\begin{minipage}{.45\linewidth}
\begin{tabular}{p{0.6cm}ccc}
  \multicolumn{4}{c}{Setting 2}\\
  \hline
$n$ & $H_{02}$ & $H_{03}$ & $H_{04}$ \\ 
  \hline
216 & 1.000 & \cellcolor{black!15} 0.054 & 0.122 \\ 
  512 & 1.000 & \cellcolor{black!15} 0.051 & 0.131 \\ 
   \hline
\end{tabular}
\end{minipage}
\begin{minipage}{.45\linewidth}
\begin{tabular}{p{0.6cm}ccc}
  \multicolumn{4}{c}{Setting 3}\\
  \hline
$n$ & $H_{01}$ & $H_{02}$ & $H_{03}$ \\  
  \hline
36 & 1.000 & \cellcolor{black!15} 0.051 & 0.102 \\ 
  64 & 1.000 & \cellcolor{black!15} 0.053 & 0.124 \\ 
   \hline
\end{tabular}
\end{minipage}
\begin{minipage}{.45\linewidth}
\begin{tabular}{p{0.6cm}ccc}
  \multicolumn{4}{c}{Setting 4}\\
  \hline
$n$ & $H_{01}$ & $H_{02}$ & $H_{03}$ \\  
  \hline
36 & 0.059 & \cellcolor{black!15} 0.091 & 0.211 \\ 
  64 & 0.058 & \cellcolor{black!15} 0.093 & 0.261 \\ 
   \hline
\end{tabular}
\end{minipage}
\end{table}

We conclude with a brief application of the procedure to the \texttt{phoneme} data set in the R-package \texttt{ElemStatLearn} \citep{RElemStatLearn}. The data consists of a total of $4509$ log-periodograms of length $p = 256$, each corresponding to a single utterance of one of several phonemes. For simplicity, we consider only the phoneme ``sh'' and, moreover, take only the first utterances of it by the first 64 speakers in the data set. This yields a data matrix with the dimensions $n = 64$ and $p = 256$, meaning that the experiment can be embedded, for example, to either of the regimes $p_n = 4n$ and $p_n = n^{4/3}$. To gain some idea on the possible Gaussianity of the data, we ran separate univariate Shapiro-Wilk tests for each of the $p$ variables using the Bonferroni correction and the significance level 0.05. Based on the tests, 4 out of the 256 variables were deemed as non-normal, implying that the assumption of Gaussianity might indeed be warranted in the current context. 

We then applied the forward testing estimator \eqref{eq:practical_estimator} with $\alpha = 0.05$ to the data and obtained the estimate $\hat{d} = 14$, implying that there is indeed great room for dimension reduction in the data set. As an alternative, ``naive'' approach we also considered forward testing based on a sequence of tests of the form \eqref{eq:low_dim_convergence} that assume $p$ to be finite. It turned out that each of the tests was rejected (with $\alpha = 0.05$), giving the maximal estimate $\hat{d} = \min\{n, p\} = 64$. As the sample size is most likely too small for the finite-dimension asymptotics to kick in (unlike for the high-dimensional asymptotics, which are in Table~\ref{tab:results_1} seen to be good approximations already for sample sizes and dimensions comparable to the current situation), we conclude that ignoring the high-dimensional nature of the data led to a gross overestimation of the latent dimension.

\section{Discussion}\label{sec:discussion}

In this short note, we showed that a classical test of subsphericity is valid also in the less often studied high-dimensional Gaussian regimes where the concentration $\gamma$ is allowed to take the extreme values $0$ and $\infty$, as long as the spikes themselves diverge to infinity. The case $\gamma = \infty$ further requires the condition that $p_n/(n \sqrt{\lambda_{nd}}) \rightarrow 0$, limiting the growth rate of the dimension $p_n$ in terms of the signal strength $\lambda_{nd}$. And even though, by our simulation study, it seems plausible that the test could be extended outside of this condition, several key arguments in our proof of Theorem~\ref{theo:goes_to_zero} hinge on it, meaning that any extensions should use a different technique of proof. 

Additionally, we derived sufficient conditions for the consistent estimation of the latent dimension $d$ with the forward testing procedure that chains together tests for the hypotheses $H_{00}, H_{01}, \ldots$. While the conditions are rather natural, again requiring that $p_n$ does not grow too fast compared to $\lambda_{nd}$, our simulation study gives indication that there is still room for improvement.

Finally, the main limiting factor of the presented results is the assumption of Gaussianity. This requirement could possibly be weakened by showing that the so-called \textit{universality phenomenon} applies to our scenario; in high-dimensional statistics, a result derived under the Gaussian assumption  is said to exhibit universality if it continues to hold when the normal distribution is replaced with some other distribution that is close to it in some suitable sense, see \cite{johnstone2018pca} for a review of such results. In the current situation concerning the limiting behavior of second-order quantities, it seems reasonable to conjecture that our main results continue to hold if the normal distribution is replaced with a distribution that shares its first four moments with the normal distribution. While the actual theoretical study of this claim goes beyond the scope of the current work (our proofs rely heavily on several pre-existing results for Wishart matrices), we nevertheless did quick experiments in Settings 1-4 described in Section \ref{sec:simulation}, with the normal distribution replaced by the symmetric Laplace mixture $(1/2) \mathcal{L}(-\mu, b) + (1/2) \mathcal{L}(\mu, b)$ having the dispersion parameter $b = \sqrt{3/2} - 1$ and the mean $\mu = \sqrt{1 - 2 b^2}$. The resulting distribution then has identical moments with the standard normal up to the fourth one. The resulting rejection rates are shown in Table \ref{tab:results_2} and they indeed match very closely with those in Table \ref{tab:results_1}, giving plausibility to the universality claim.

\begin{table}[h]
\caption{The subtables give the observed rejection rates for different null hypotheses over 10000 independent replicates under each of the four settings when the data are drawn from the symmetric Laplace mixture. The columns corresponding to the true dimension are shaded grey.} 
\label{tab:results_2}
\centering
\begin{minipage}{.45\linewidth}
\begin{tabular}{p{0.6cm}ccc}
  \multicolumn{4}{c}{Setting 1}\\
  \hline
$n$ & $H_{02}$ & $H_{03}$ & $H_{04}$ \\ 
  \hline
216 & 1.000 & \cellcolor{black!15} 0.055 & 0.121 \\ 
  512 & 1.000 & \cellcolor{black!15} 0.054 & 0.137 \\ 
   \hline
\end{tabular}
\end{minipage}
\begin{minipage}{.45\linewidth}
\begin{tabular}{p{0.6cm}ccc}
  \multicolumn{4}{c}{Setting 2}\\
  \hline
$n$ & $H_{02}$ & $H_{03}$ & $H_{04}$ \\ 
  \hline
  216 & 1.000 & \cellcolor{black!15} 0.055 & 0.124 \\ 
  512 & 1.000 & \cellcolor{black!15} 0.054 & 0.138 \\ 
   \hline
\end{tabular}
\end{minipage}
\begin{minipage}{.45\linewidth}
\begin{tabular}{p{0.6cm}ccc}
  \multicolumn{4}{c}{Setting 3}\\
  \hline
$n$ & $H_{01}$ & $H_{02}$ & $H_{03}$ \\  
  \hline
36 & 1.000 & \cellcolor{black!15} 0.052 & 0.109 \\ 
  64 & 1.000 & \cellcolor{black!15} 0.049 & 0.128 \\ 
   \hline
\end{tabular}
\end{minipage}
\begin{minipage}{.45\linewidth}
\begin{tabular}{p{0.6cm}ccc}
  \multicolumn{4}{c}{Setting 4}\\
  \hline
$n$ & $H_{01}$ & $H_{02}$ & $H_{03}$ \\  
  \hline
36 & 0.057 & \cellcolor{black!15} 0.091 & 0.213 \\ 
  64 & 0.061 & \cellcolor{black!15} 0.090 & 0.253 \\ 
   \hline
\end{tabular}
\end{minipage}
\end{table}


\appendix

\section{Discussion of Theorem 4 in \cite{schott2006high}}\label{sec:schott}

For convenience, this section uses the notation of \cite{schott2006high}. We first show that the final part of Condition 2 in \cite{schott2006high}, assuming that $\lim_{k \rightarrow \infty} \lambda_{i, k}/\mathrm{tr}(\Sigma_k) = \rho_i \in (0, 1)$, $i = 1, \ldots , q$, and that $\sum_{i=1}^q \rho_i \in (0, 1)$, is actually not necessary for their Theorem 4. This condition is used both in equation (22) and in the equation right after (24) to guarantee that $m/\lambda_q = \mathcal{O}(1)$. This, in conjunction with the observation that $\mathrm{tr}(W_{12} W_{12}') = o_p(m)$, then gives the relation,
\begin{align*}
    \frac{1}{\lambda_q}\mathrm{tr}(W_{12} W_{12}') = \frac{m}{\lambda_q} o_p(1) = o_p(1),
\end{align*}
used in bounding the moments. However, the same relation follows directly from the divergence of the spike eigenvalues $\lambda_j$ by first observing that, by the proof of our Lemma \ref{lem:wishart_2}, we have $\mathrm{tr}(W_{12} W_{12}') = q c + o_p(1) $, where $c \in (0, \infty)$ is the limit of $p/n$. Note also that, to obtain the final bound in the equation right after (24) without assuming anything about the relative growth rates of the spikes, we use the bound $\| \Sigma_*^{-1} \| \leq \lambda_q^{-1} \lambda_q \mathrm{tr}(\Sigma_*^{-1}) \leq \lambda_q^{-1} q$ (which is valid simply by the ordering of the spike eigenvalues). Thus, the result of Theorem~4 can be obtained without the final part of Condition~2 in \cite{schott2006high}.

Additionally, we remark that, by what appears to be an oversight, the proof of Theorem 4 in \cite{schott2006high} does not hold as such in the case where $p/n \rightarrow c = 1$. Namely, in equation (22) and in the equation right after (24), the upper bounds involve the term $\phi_r^{-1}(S_{22 \cdot 1})$, which converges in probability to $(1 - c^{1/2})^{-2}$ which fails to be finite when $c = 1$. It seems to us that introducing some additional (non-trivial) assumptions on the spike eigenvalues could possibly recover the proof for $c = 1$ as, indeed, the simulations in \cite{schott2006high} suggest that the result of Theorem 4 holds in that case also. 

\section{Proofs}

Before the proof of Theorem \ref{theo:goes_to_zero} we establish an auxiliary lemma.

\begin{lemma}\label{lem:wishart_2}
    Let $W_n \sim \mathcal{W}_{p_n}(I_{p_n}/n, n)$ be partitioned as
    \begin{align*}
    W_n = \begin{pmatrix}
    W_{n,11} & W_{n,12} \\
    W_{n,21} & W_{n,22}
    \end{pmatrix},
    \end{align*}
    where the block $W_{n,11}$ has the size $d \times d$ and $\mathcal{W}_p(\Sigma, \nu)$ denotes the $(p \times p)$-dimensional Wishart distribution with the scale matrix $\Sigma$ and $\nu$ degrees of freedom. Then, as $n, p_n \rightarrow \infty$,
    \begin{enumerate}
        \item if $p_n/n \rightarrow 0$, we have,
            \begin{align*}
        \mathrm{tr}( W_{n,12} W_{n,21} ) = \mathcal{O}_p\left( \frac{p_n}{n} \right),
    \end{align*}
        \item if $p_n/n \rightarrow \infty $ and $p_n/(n \sqrt{\lambda_{n}}) \rightarrow 0$ for some sequence $\lambda_n \rightarrow \infty$, we have,
        \begin{align*}
        \mathrm{tr}( W_{n,12} W_{n,21} ) = o_p\left( \sqrt{ \lambda_{n} } \right).
    \end{align*}
    \end{enumerate}
    
\end{lemma}

\begin{proof}[Proof of Lemma \ref{lem:wishart_2}]
    The matrix $ W_n $ has the same distribution as the (biased) non-centered sample covariance matrix of a random sample $ z_1, \ldots , z_n $ from the $ p_n $-variate standard normal distribution. Hence, letting $Y_n := W_{n,12} W_{n,21} $, we have, for arbitrary $ j = 1, \ldots , d $, that
\begin{align*}
    y_{n,jj} = \sum_{k = d + 1}^{p_n} \left( \frac{1}{n} \sum_{i = 1}^{n} z_{ij} z_{ik} \right)^2
\end{align*} 
The expected value of $ y_{n,jj} $ is
\begin{align*} 
\mathbb{E} (y_{n,jj}) = \frac{1}{n^2}  \sum_{k = d + 1}^{p_n}\sum_{i = 1}^{n} \sum_{\ell = 1}^{n}  \mathbb{E} (z_{ij} z_{ik} z_{\ell j} z_{\ell k})
=  \frac{p_n - d}{n}.
\end{align*}
Whereas, its second moment is 
\begin{align*} 
\mathbb{E} (y_{n,jj}^2) = & \frac{1}{n^4} \sum_{k = d + 1}^{p_n} \sum_{k' = d + 1}^{p_n} \sum_{i = 1}^{n} \sum_{i' = 1}^{n} \sum_{\ell = 1}^{n} \sum_{\ell' = 1}^{n}  \mathbb{E} (z_{ij} z_{ik} z_{\ell j} z_{\ell k} z_{i'j} z_{i'k'} z_{\ell' j} z_{\ell' k'}) \\
= & \frac{1}{n^4} \sum_{k k' i i' \ell \ell'} \mathbb{E} (z_{ij} z_{\ell j} z_{i'j} z_{\ell' j} ) \mathbb{E} (z_{ik} z_{\ell k} z_{i'k'} z_{\ell' k'} ) \\
= & \frac{1}{n^4} \sum_{k k' i i' \ell \ell'} (\delta_{i \ell} \delta_{i' \ell'} + \delta_{i i'} \delta_{\ell \ell'} + \delta_{i \ell'} \delta_{\ell i'}) (\delta_{i \ell} \delta_{i' \ell'} + \delta_{i i'} \delta_{k k'} \delta_{\ell \ell'} + \delta_{i \ell'} \delta_{k k'} \delta_{\ell i'}) \\
= & \frac{p_n - d}{n^3} \{ (p_n - d) n + 2 (p_n - d) + 2 n + 4 \},
\end{align*}
where the second-to-last equality uses Isserlis' theorem. Consequently, the variance of $ y_{n,jj} $ is
\[ 
\mathrm{Var}(y_{n,jj}) = \frac{2(p_n - d)}{n^3} \{ 2 + (p_n - d) + n \}. 
\]
Hence, the moments of $ t_{n,jj} := (n/p_n) y_{n,jj} $ are $\mathrm{E}(t_{n,jj}) = 1 - d/p_n = 1 + o(1)$ and
\begin{align*}
    \mathrm{Var}(t_{n,jj}) = 2\left(1 - \frac{d}{p_n} \right) \left\{ \frac{2}{n^2} + \frac{p_n - d}{n^2} + \frac{1}{n} \right\}  = o(1).
\end{align*}
The first claim now follows and the second one is straightforwardly verified to be true in a like manner.
\end{proof}

\begin{proof}[Proof of Theorem \ref{theo:goes_to_zero}]
Due to centering we may WLOG assume that $\mu_n = 0$ for all $n \in \mathbb{N}$. Moreover, as our main claim depends on $S_n$ only through its eigenvalues, we may, again WLOG, assume that $\Sigma_n = \mathrm{diag}(\lambda_{n1}, \ldots , \lambda_{nd}, \sigma^2, \ldots , \sigma^2)$. Finally, as the left-hand side of our main claim is invariant under scaling of the observations, we may WLOG assume that $\sigma^2 = 1$. 

Denoting $n_0 := n - 1$, we have that $S_n = \Sigma_n^{1/2} W_n \Sigma_n^{1/2} $ where $W_n \sim \mathcal{W}_{p_n}\{ n_0^{-1} I_{p_n}, n_0 \}$ is the sample covariance matrix of a sample of size $n$ from the $p_n$-variate standard normal distribution. Denote then $\Lambda_n = \mathrm{diag}(\lambda_{n1}, \ldots , \lambda_{nd})$ and partition $S_n$ and $W_n$ as
\begin{align*}
    S_n = \begin{pmatrix}
    S_{n,11} & S_{n,12} \\
    S_{n,21} & S_{n,22}
    \end{pmatrix}
    = 
    \begin{pmatrix}
    \Lambda_n^{1/2} W_{n,11} \Lambda_n^{1/2} & \Lambda_n^{1/2} W_{n,12} \\
    W_{n,21} \Lambda_n^{1/2} & W_{n,22}
    \end{pmatrix}
\end{align*}
where the matrices $S_{n, 11}$ and $W_{n,11}$ are of the size $d \times d$. Then $W_{n,22} \sim \mathcal{W}_{r_n}\{ n_0^{-1} I_{r_n}, n_0 \}$, where $r_n := p_n - d$, and the Schur complement $S_{n, 22 \cdot 1}$ satisfies
\begin{align*}
    S_{n, 22 \cdot 1} := S_{n, 22} - S_{n, 21} S_{n, 11}^{-1} S_{n, 12} = W_{n, 22 \cdot 1} \sim \mathcal{W}_{r_n}\{ n_0^{-1} I_{r_n}, n_0 - d \},
\end{align*}
where the distribution of $W_{n, 22 \cdot 1}$ follows from Theorem 3.4.6 in \cite{mardia1979multivariate}. Consequently, $G_n := \{ n_0/(n_0 - d) \} S_{n, 22 \cdot 1} \sim \mathcal{W}_{r_n}\{ (n_0 - d)^{-1} I_{r_n}, n_0 - d \}$, implying that $m_{2,r_n}(S_{n, 22 \cdot 1})/ m_{1,r_n}(S_{n, 22 \cdot 1})^2 = m_{2,r_n}(G_n)/ m_{1,r_n}(G_n)^2$. Hence, by Theorem 3.7 in \cite{birke2005note}, we have
\begin{align}\label{eq:schur_complement_distribution}
    (n - d - 1) \left\{ \frac{m_{2,r_n}(S_{n, 22 \cdot 1})}{m_{1,r_n}(S_{n, 22 \cdot 1})^2} - 1 \right\} - r_n \rightsquigarrow \mathcal{N}(1, 4),
\end{align}
regardless of which of the two asymptotic regimes we are in. Note also that, as $G_n$ is of the size $r_n \times r_n$, the notation $m_{k, r_n}(G_n)$ simply refers to the $k$th sample moment of its eigenvalues.

Consider next the regime where $p_n /n \rightarrow 0 $ and assume that
\begin{align}\label{eq:eigenvalue_moment_convergence}
    m_{k,r_n}(S_{n, 22 \cdot 1}) = m_{k,r_n}(S_{n}) + o_p(1/n),
\end{align}
for $k = 1, 2$. Then, the difference
\begin{align}\label{eq:test_difference}
\begin{split}
    & (n - d - 1) \left\{ \frac{m_{2,r_n}(S_n)}{m_{1,r_n}(S_n)^2} - \frac{m_{2,r_n}(S_{n, 22 \cdot 1})}{m_{1,r_n}(S_{n, 22 \cdot 1})^2} \right\}\\
    =& (n - d - 1) \frac{m_{2,r_n}(S_n) m_{1,r_n}(S_{n, 22 \cdot 1})^2 - m_{2,r_n}(S_{n, 22 \cdot 1}) m_{1,r_n}(S_n)^2}{m_{1,r_n}(S_{n, 22 \cdot 1})^2 m_{1,r_n}(S_n)^2},
\end{split}
\end{align}
is easily checked to be of the order $o_p(1)$ using \eqref{eq:eigenvalue_moment_convergence} and the results following from Section 2 in \cite{birke2005note} that $m_{1,r_n}(S_{n, 22 \cdot 1}) \rightarrow_p 1$ and $m_{2,r_n}(S_{n, 22 \cdot 1}) \rightarrow_p 1$. Hence, the first claim of the theorem follows from \eqref{eq:schur_complement_distribution}.

Similarly, in the regime that $p_n /n \rightarrow \infty $ and $p_n/(n \sqrt{\lambda_{nd}}) \rightarrow 0$, assume that
\begin{align}\label{eq:eigenvalue_moment_convergence_2}
\begin{split}
    m_{1,r_n}(S_{n, 22 \cdot 1}) &= m_{1,r_n}(S_{n}) + o_p(1/p_n),\\
    m_{2,r_n}(S_{n, 22 \cdot 1}) &= m_{2,r_n}(S_{n}) + o_p(1/n).
\end{split}
\end{align}
Then, the difference \eqref{eq:test_difference} can similarly be shown to be of the order $o_p(1)$ (proving the second claim of the theorem). Note that in this case we require a faster convergence from the first moment since, by Section 2 of \cite{birke2005note} we have again $m_{1,r_n}(S_{n, 22 \cdot 1}) \rightarrow_p 1$ but the second moment behaves as $m_{2,r_n}(S_{n, 22 \cdot 1}) - (n_0 - d)(r_n + 1)/n_0^2 \rightarrow_p 1$

Thus, we next establish \eqref{eq:eigenvalue_moment_convergence} for $k = 1, 2$ and \eqref{eq:eigenvalue_moment_convergence_2}, starting from the former. As $p_n/n \rightarrow 0$, we may without loss of generality assume $n > p_n$, implying that $S_n$ is almost surely positive definite. Now, we have for $ S_{n, 11 \cdot 2} := S_{n, 11} -  S_{n, 12} S_{n, 22}^{-1} S_{n, 21} $ that,
\begin{align}\label{eq:inequality_chain}
\begin{split}
    & \phi_d^{-1}(S_{n, 11 \cdot 2})\\
    =& \phi_1\{(S_{n, 11} -  S_{n, 12} S_{n, 22}^{-1} S_{n, 21})^{-1}\}\\
    =& \phi_1(S_{n, 11}^{-1} + S_{n, 11}^{-1} S_{n, 12} S_{n, 22 \cdot 1}^{-1} S_{n, 21} S_{n, 11}^{-1}) \\
    \leq& \phi_1(\Lambda_n^{-1/2} W_{n,11}^{-1} \Lambda_n^{-1/2}) + \phi_1(\Lambda_n^{-1/2} W_{n, 11}^{-1} W_{n, 12} W_{n, 22 \cdot 1}^{-1} W_{n, 21} W_{n, 11}^{-1} \Lambda_n^{-1/2}) \\
    \leq& \phi^2_1(\Lambda_n^{-1/2}) \phi_1(W_{n,11}^{-1}) + \phi^2_1(\Lambda_n^{-1/2}) \phi_1^2(W_{n,11}^{-1}) \phi_1(W_{n, 12} W_{n, 22 \cdot 1}^{-1} W_{n, 21}),
\end{split}
\end{align}
where the second equality follows from the Woodbury matrix identity, the first inequality uses Weyl's inequality and the second inequality follows from the sub-multiplicativity of the spectral norm. Now, Assumption \ref{assu:eigenvalues} guarantees that $ \phi^2_1(\Lambda_n^{-1/2}) = \lambda_{nd}^{-1} \rightarrow 0$ and, since $W_{n, 11} \rightarrow_p I_d$, we further have, by the continuity of eigenvalues, that $\phi_1(W_{n,11}^{-1}) \rightarrow_p 1$. Write then,
\begin{align*}
    \phi_1(W_{n, 12} W_{n, 22 \cdot 1}^{-1} W_{n, 21}) &= \| W_{n, 12} W_{n, 22 \cdot 1}^{-1} W_{n, 21} \|_2\\
    &\leq \| W_{n, 12}  \|_2^2 \| W_{n, 22 \cdot 1}^{-1} \|_2\\
    &\leq \mathrm{tr}( W_{n, 12} W_{n, 21} ) \phi_{r_n}^{-1}( W_{n, 22 \cdot 1} ),
\end{align*}
where $\| \cdot \|_2$ denotes the spectral norm. Now, since $ G_n = \{ n_0/(n_0 - d) \} W_{n, 22 \cdot 1} \sim \mathcal{W}_{r_n}\{ (n_0 - d)^{-1} I_{r_n}, n_0 - d \} $, we have by the discussion after Theorem 1.1 in \cite{rudelson2009smallest} that
\begin{align*}
    \mathbb{P}\left\{\phi_{r_n}(G_n) \leq \left(1 -\sqrt{\frac{r_n}{n_0 - d}} - \frac{t}{\sqrt{n_0 - d}}\right)^2\right\} \leq e^{-t^2/2},
\end{align*}
for all $t > 0$. Substituting $t = (1/2) \sqrt{n_0 - d} - \sqrt{r_n}$ (which is positive for a large enough $n$), gives,
\begin{align*}
    \mathbb{P}\left\{\phi_{r_n}(G_n) \leq  1/4 \right\} \leq \exp[-\{ (1/2) \sqrt{n_0 - d} - \sqrt{r_n}\}^2/2] \rightarrow 0.
\end{align*}
Hence,
\begin{align}\label{eq:bounded_final_inverse}
    \phi_{r_n}^{-1}( W_{n, 22 \cdot 1} ) = \frac{n_0}{n_0 - d} \frac{1}{\{\phi_{r_n}(G_n) - 1/4\} + 1/4} = \mathcal{O}_p(1),
\end{align}
where the final step follows as $\phi_{r_n}(G_n) - 1/4$ is positive with probability approaching one. Finally, by Lemma \ref{lem:wishart_2}, we have that $\mathrm{tr}( W_{n,12} W_{n,21} ) = \mathcal{O}_p( p_n/n ) = o_p(1)$ and plugging all these in to \eqref{eq:inequality_chain}, we obtain that $ 0 < \phi_d^{-1}(S_{n, 11 \cdot 2}) \leq o_p(1) $ (where the first inequality holds a.s. by the positive-definiteness of the Schur complement). This, in conjunction with the fact that $\phi_1(S_{n, 22 \cdot 1}) \rightarrow_p 1$, implied by Theorem 2 in \cite{karoui2003largest}, lets us to conclude that $ \mathbb{P}\{ \phi_d(S_{n, 11 \cdot 2}) > \phi_1(S_{n, 22 \cdot 1}) \} \rightarrow 1 $ as $n \rightarrow \infty$, and, in the sequel, we restrict our attention to this event, allowing us to apply Theorem 3 in \cite{schott2006high}, equation (17) of which yields,
\begin{align}\label{eq:schott_inequality_1}
\begin{split}
    0 &\leq m_{1,r_n}(S_{n, 22 \cdot 1}) - m_{1,r_n}(S_{n}) \\
    &\leq \frac{\phi^2_1(S_{n, 22 \cdot 1})}{r_n \{\phi^{-1}_1(S_{n, 22 \cdot 1}) - \phi^{-1}_d(S_{n, 11 \cdot 2}) \}}  \mathrm{tr}(S_{n, 11}^{-1} S_{n, 12} S_{n, 22 \cdot 1}^{-2} S_{n, 21} S_{n, 11}^{-1}) \\
    &\leq r_n^{-1} \{ 1 + o_p(1) \} \| \Lambda_n^{-1/2} \|^2 \| W_{n, 11}^{-1} \|^2 \| W_{n, 12} W_{n, 22 \cdot 1}^{-2} W_{n, 21} \| \\
    &= o_p(1/p_n) \| W_{n, 12} W_{n, 22 \cdot 1}^{-2} W_{n, 21} \|.
\end{split}
\end{align}
Let the singular value decomposition of $W_{n, 21}$ be $ W_{n, 21} = R_n D_n T_n' $. Then,
\begin{align*}
    \| W_{n, 12} W_{n, 22 \cdot 1}^{-2} W_{n, 21} \|^2 &\leq \| D_n \|^4 \| R_n' W_{n, 22 \cdot 1}^{-2} R_n \|^2 \\
    &= \{ \mathrm{tr}(W_{n, 12} W_{n, 21}) \}^2 \| R_n' W_{n, 22 \cdot 1}^{-2} R_n \|^2 \\
    &= \mathcal{O}_p(p_n^2/n^2) \sum_{j = 1}^d \phi_j^2(R_n' W_{n, 22 \cdot 1}^{-2} R_n) \\
    &\leq \mathcal{O}_p(p_n^2/n^2) \sum_{j = 1}^d \phi_j^2(W_{n, 22 \cdot 1}^{-2}) \\
    &\leq \mathcal{O}_p(p_n^2/n^2) d \phi_1^2(W_{n, 22 \cdot 1}^{-2})\\
    &= \mathcal{O}_p(p_n^2/n^2) \phi^{-4}_{r_n}(W_{n, 22 \cdot 1})\\
    &= \mathcal{O}_p(p_n^2/n^2),
\end{align*}
where the second equality follows from Lemma \ref{lem:wishart_2}, the second inequality from the Poincar\'e separation theorem and the final equality from \eqref{eq:bounded_final_inverse}. Plugging this in to \eqref{eq:schott_inequality_1} then establishes \eqref{eq:eigenvalue_moment_convergence} for $k=1$.

To show the same for $k = 2$, we apply equation (18) from Theorem 3 in \cite{schott2006high} to obtain
\begin{align*}
    0 &\leq m_{2,r_n}(S_{n, 22 \cdot 1}) - m_{2,r_n}(S_{n}) \\
    &\leq \frac{2 \phi^4_1(S_{n, 22 \cdot 1})}{r_n}  \left\{ 1 + \frac{ \phi^{-1}_d(S_{n, 11 \cdot 2})}{\phi^{-1}_1(S_{n, 22 \cdot 1}) - \phi^{-1}_d(S_{n, 11 \cdot 2})} \right\} \mathrm{tr}(S_{n, 11}^{-1} S_{n, 12} S_{n, 22 \cdot 1}^{-2} S_{n, 21} S_{n, 11}^{-1}),
\end{align*}
where arguing as in the case $k = 1$ shows that the right-hand side is bounded by a $o_p(1/n)$-quantity, concluding the proof of the case where $p_n/n \rightarrow 0$.

For the second claim, we, without loss of generality, assume that $p_n > n$, implying that the rank of $S_{n, 22 \cdot 1}$ is almost surely $n_0 - d$. Denote then any of its eigendecompositions by $ S_{n, 22 \cdot 1} = Q_n \Delta_n Q_n' $ where $Q_n$ is a $r_n \times (n_0 - d)$ matrix with orthonormal columns and $\Delta_n$ contains the almost surely positive $n_0 - d$ eigenvalues. Our aim is to use Corollary 3 of \cite{schott2006high} and, for that, we first show that $ \mathbb{P}\{ \phi_d(\tilde{S}_{n, 11 \cdot 2}) > \phi_1(S_{n, 22 \cdot 1}) \} \rightarrow 1 $ as $n \rightarrow \infty$, where $\tilde{S}_{n, 11 \cdot 2} := S_{n, 11} -  S_{n, 12} Q_n (Q_n' S_{n, 22} Q_n)^{-1} Q_n' S_{n, 21}$. Now, the inverse of $ \tilde{S}_{n, 11 \cdot 2} $ is $ S_{n, 11}^{-1} + S_{n, 11}^{-1} S_{n, 12} Q_n \Delta_n^{-1} Q_n' S_{n, 21} S_{n, 11}^{-1} $ and, proceeding as in \eqref{eq:inequality_chain}, we see that $\phi_d^{-1}(\tilde{S}_{n, 11 \cdot 2})$ has the upper bound,
\begin{align*}
    \phi^2_1(\Lambda_n^{-1/2}) \phi_1(W_{n,11}^{-1}) + \phi^2_1(\Lambda_n^{-1/2}) \phi_1^2(W_{n,11}^{-1}) \phi_1(W_{n, 12} Q_n \Delta_n^{-1} Q_n' W_{n, 21}),
\end{align*}
where the final leading eigenvalue has, by the Poincar\'e separation theorem, the upper bound $ \mathrm{tr}(W_{n, 12} W_{n, 21}) \phi^{-1}_{n_0 - d} (\Delta_n) $. Now, as in the proof of the first claim, \cite{rudelson2009smallest} can be used to show that $\phi^{-1}_{n_0 - d} (\Delta_n) = \mathcal{O}_p(n/p_n)$. Furthermore, Lemma \ref{lem:wishart_2} shows that $ \mathrm{tr}(W_{n, 12} W_{n, 21}) = o_p(\sqrt{\lambda_{nd}}) $ under our assumptions, finally yielding that,
\begin{align*}
    \phi_d^{-1}(\tilde{S}_{n, 11 \cdot 2}) \leq \frac{1}{\lambda_{nd}} \{ 1 + o_p(1) \} + o_p\{n/(p_n \sqrt{\lambda_{nd}}) \},
\end{align*}
This, in conjunction with the result that $(p_n/n)\phi_1^{-1}(S_{n, 22 \cdot 1}) \rightarrow_p 1$, implied by Theorem 1 in \cite{karoui2003largest}, guarantees now that
\begin{align*}
    (p_n/n)\{\phi_1^{-1}(S_{n, 22 \cdot 1}) - \phi_d^{-1}(\tilde{S}_{n, 11 \cdot 2}) \} \geq 1 - \frac{p_n}{n \lambda_{nd}} \{ 1 + o_p(1) \} + o_p(1) = 1 + o_p(1),
\end{align*}
showing that $ \mathbb{P}\{ \phi_d(\tilde{S}_{n, 11 \cdot 2}) > \phi_1(S_{n, 22 \cdot 1}) \} \rightarrow 1 $, as desired, and allowing us to restrict our attention to the corresponding set and to use Corollary 3 in \cite{schott2006high}. Its first part gives us
\begin{align}\label{eq:schott_inequality_3}
\begin{split}
    0 &\leq m_{1,r_n}(S_{n, 22 \cdot 1}) - m_{1,r_n}(S_{n}) \\
    &\leq \frac{\phi^2_1(S_{n, 22 \cdot 1})}{r_n \{\phi^{-1}_1(S_{n, 22 \cdot 1}) - \phi^{-1}_d(\tilde{S}_{n, 11 \cdot 2}) \}}  \mathrm{tr}(S_{n, 11}^{-1} S_{n, 12} Q_n \Delta_n^{-2} Q_n' S_{n, 21} S_{n, 11}^{-1}) \\
    &\leq \frac{p_n^3}{n^3 r_n \lambda_{nd}} \lambda_{nd} \mathrm{tr}(\Lambda_n^{-1}) \{ d + o_p(1) \}  \| W_{n, 12} Q_n \Delta_n^{-2} Q_n' W_{n, 21} \|,
\end{split}
\end{align}
where $ \lambda_{nd} \mathrm{tr}(\Lambda_n^{-1}) \leq d$. Reasoning similarly as with the first claim of the theorem, we further have
\begin{align*}
    \| W_{n, 12} Q_n \Delta_n^{-2} Q_n' W_{n, 21} \| &\leq \mathrm{tr}(W_{n, 12} W_{n, 21}) \| R_n' Q_n \Delta_n^{-2} Q_n' R_n \|\\
    &\leq o_p(\sqrt{\lambda_{nd}}) \mathcal{O}_p(n^2/p_n^2),
\end{align*}
where $R_n$ again contains the left singular vectors of $W_{n,21}$. Plugging the obtained upper bound to \eqref{eq:schott_inequality_3} then finally gives the first claim of \eqref{eq:eigenvalue_moment_convergence_2} and the second claim is obtained in exactly the same manner but by using the second inequality of Corollary 3 in \cite{schott2006high} instead of the first.
\end{proof}

\begin{proof}[Proof of Theorem \ref{theo:power_for_small_k}]

We begin with the case \textit{i)} and assume first that $\gamma = 0$. The test statistics $g_{n,k} = (n - k - 1) T_{n, k} - (p_n - k)$, $k = 0, \ldots , d - 1$, then satisfy,
\begin{align}\label{eq:gnk_lower_bound}
\begin{split}
    g_{n,k} \geq& (n - d - 1) T_{n, k} - (p_n - d) + (k - d)\\
            =& g_{n, d} + (k - d) + (n - d - 1)(T_{n, k} - T_{n, d}).
\end{split}
\end{align}
We derive a lower bound for the term $T_{n, k} - T_{n, d}$, using the shorthand notations $m_r := m_{r,p_n-d}(S_n) \geq 0$, for $r = 1, 2$, $b_{1k} := \sum_{j = k + 1}^d \phi_j(S_n)$, $b_{2k} := \sum_{j = k + 1}^d \phi_j(S_n)^2$ and $r_n := p_n - d$:
\begin{align}\label{eq:T_decomposition}
\begin{split}
    T_{n, k} - T_{n, d} =& (r_n + d - k) \frac{r_n m_2 + b_{2k}}{(r_n m_1 + b_{1k})^2} - \frac{m_2}{m_1^2} \\
    \geq& \frac{m_2 + b_{2k}/r_n}{(m_1 + b_{1k}/r_n)^2} - \frac{m_2}{m_1^2}\\
    =& \frac{m_1^2 b_{2k}/r_n - 2 m_1 m_2 b_{1k}/r_n - m_2 b_{1k}^2/r_n^2}{(m_1 + b_{1k}/r_n)^2 m_1^2}.
\end{split}
\end{align}
Denoting now $c_{nk} := b_{1k}/r_n \geq 0$, the RHS of \eqref{eq:T_decomposition} splits into three terms, of which the second one satisfies
\begin{align*}
    \frac{ -2 m_2 c_{nk}}{(m_1 + c_{nk})^2 m_1} = -2 \frac{m_2}{m_1} \left\{ \frac{1}{m_1 + c_{nk}} - \frac{m_1}{(m_1 + c_{nk})^2} \right\} \geq -2 \frac{m_2}{m_1^2},  
\end{align*}
where the RHS is $-2 + o_p(1)$ as $m_1, m_2 \rightarrow_p 1 $ by \eqref{eq:eigenvalue_moment_convergence} and the formulas in Section 2 of \cite{birke2005note}. Similarly, the third term on the RHS of \eqref{eq:T_decomposition} has
\begin{align*}
    \frac{ -m_2 c_{nk}^2}{(m_1 + c_{nk})^2 m_1^2} \geq -\frac{m_2}{m_1^2} = - 1 + o_p(1). 
\end{align*}
Observe then that the power mean inequality states that $b_{2k} \geq (d - k)^{-1} b_{1k}^2$. Thus, we have for the first term on the RHS of \eqref{eq:T_decomposition} that, 
\begin{align*}
    \frac{b_{2k}/r_n}{(m_1 + b_{1k}/r_n)^2} \geq \frac{(d - k)^{-1} b_{1k}^2/r_n}{(m_1 + b_{1k}/r_n )^2} = \frac{(d - k)^{-1} }{ (m_1 \sqrt{r_n}/b_{1k} + 1/\sqrt{r_n} )^2}.
\end{align*}
Now, $b_{1k} \geq (d - k) \phi_d(S_n) \geq (d - k) \phi_d(S_{n, 11})$ where the second inequality uses Lemma 2 in \cite{schott2006high}. Moreover, as the spectra of $RT$ and $TR$ are equal for any two square matrices $T, R$, Lemma 3 in \cite{schott2006high} gives $\phi_d(S_{n, 11}) = \phi_d(\Lambda_n^{1/2} W_{n, 11} \Lambda_n^{1/2}) = \phi_d(\Lambda_n W_{n, 11}) \geq \lambda_{nd} \phi_d(W_{n, 11})$, a.s., where $\phi_d(W_{n, 11}) \rightarrow_p 1$ by the continuity of eigenvalues. Consequently,
\begin{align*}
    \frac{b_{2k}/r_n}{(m_1 + b_{1k}/r_n)^2} \geq \frac{(d - k)^{-1} }{\{(d - k)^{-1} m_1 \phi_d^{-1}(W_{n, 11}) \sqrt{r_n}/\lambda_{nd} + 1/\sqrt{r_n} \}^2},
\end{align*}
where the RHS is, by our assumption that $p_n/\lambda_{nd}^2 \rightarrow 0$, of the form  $a_1/a_{n2}^2$ for some constant $a_1 > 0$ and a sequence $a_{n2}$ of random variables such that $a_{n2} \rightarrow_p 0$.

Plugging now everything in to \eqref{eq:gnk_lower_bound} and using the fact that  $g_{n, d} = \mathcal{O}_p(1)$ (by Theorem \ref{theo:goes_to_zero}), gives,
\begin{align*}
    g_{n,k}/n \geq& g_{n, d}/n + (k - d)/n + \frac{n - d - 1}{n} \{ -3 + o_p(1) + a_{1}/a_{n2}^2 \}\\
    =& -3 + o_p(1) + a_{n3}/a_{n2}^2,
\end{align*}
where $a_{n3} := a_1 \{1 - ( d + 1 )/n \} \rightarrow a_1 > 0$.

Write now $z_n = \Omega_\infty(1)$ if a sequence of random variables $z_n$ satisfies $\mathbb{P}(z_n \leq M) \rightarrow 0$ for all $M > 0$ as $n \rightarrow \infty$. Then $a_{n3}/a_{n2}^2 = \Omega_\infty(1)$ as can be seen by writing, for an arbitrary $M > 0$,
\begin{align*}\label{eq:conditioning_op}
\begin{split}
    & \mathbb{P}(a_{n3}/a_{n2}^2 \leq M) \\
    =& \mathbb{P}(a_{n3}/a_{n2}^2 \leq M \mid a_{n3} \geq a_1/2) \mathbb{P}(a_{n3} \geq a_1/2)\\
    +& \mathbb{P}(a_{n3}/a_{n2}^2 \leq M \mid a_{n3} < a_1/2) \mathbb{P}(a_{n3} < a_1/2)\\
    \leq& \mathbb{P}\{ a_1/(2 a_{n2}^2) \leq M \mid a_{n3} \geq a_1/2\} \mathbb{P}(a_{n3} \geq a_1/2) + o(1) \\
    \leq& \mathbb{P}\{ a_{n2}^2 \geq a_1/(2M) \} + o(1)\\
    =&o(1).
\end{split}
\end{align*}
The proof for \textit{i)} is now finished once we show that $\mathcal{O}_p(1) + \Omega_\infty(1) = \Omega_\infty(1)$. Letting $a_n, b_n$ be arbitrary sequences of random variables with the orders $a_n = \mathcal{O}_p(1)$ and $b_n = \Omega_\infty(1)$, fix $\varepsilon, M > 0$ and take $C, n_0 > 0$ to be such that $\mathbb{P}(|a_n| \geq C) \leq \varepsilon/2$ for all $n > n_0$. Moreover, let $n_1$ be such that for all $n > n_1$, we have $\mathbb{P}(b_n \leq M + C) \leq \varepsilon/2$. Then, for $n > \max \{ n_0, n_1 \}$, we have
\begin{align*}
     & \mathbb{P}(a_n + b_n \leq M) \\
    =& \mathbb{P}(a_n + b_n \leq M \mid a_n \geq -C) \mathbb{P}( a_n \geq -C )\\
    +& \mathbb{P}(a_n + b_n \leq M \mid a_n < -C) \mathbb{P}( a_n < -C )\\
    \leq& \mathbb{P}(-C + b_n \leq M \mid a_n \geq -C) \mathbb{P}( a_n \geq -C ) + \varepsilon/2 \\
    \leq& \mathbb{P}(b_n \leq M + C) + \varepsilon/2\\
    \leq& \varepsilon,
\end{align*}
proving the claim.

Moving our attention to the case $\gamma \in (0, \infty)  \setminus \{ 1 \}$ of part \textit{i)} of the theorem, exactly the same proof as was used for $\gamma = 0$ suffices also here after the modification of a single part: To see that $m_2/m_1^2$ converges in probability to a constant, we use (21) from \cite{schott2006high} in conjunction with Lemma 2.2 in \cite{wang2013sphericity} to obtain $m_1 \rightarrow_p~1$ and $m_2 \rightarrow_p 1 + \gamma$.


Finally, to obtain part \textit{ii)} of the claim, we make the following modifications to the proof: To control $m_2/m_1^2$, the equation \eqref{eq:eigenvalue_moment_convergence_2} together with the formulas in Section 2 of \cite{birke2005note} give $m_1 = 1 + o_p(1)$ and $m_2 = \mathcal{O}_p(p_n/n)$. Hence, we get the following lower bound for $g_{n, k}/n$:
\begin{align*}
    & o_p(1) + \{ 1 + o(1) \} \frac{p_n}{n} \left[ \mathcal{O}_p(1) + (n/p_n) \frac{(d - k)^{-1} }{\left\{ \frac{m_1 \sqrt{r_n}}{(d - k) \phi_d(W_{n, 11}) \lambda_{nd}} + \frac{1}{\sqrt{r_n}} \right\}^2} \right] \\
    =& o_p(1) + \{ 1 + o(1) \} \frac{p_n}{n} \left\{ \mathcal{O}_p(1) + \Omega_\infty(1) \right\},
\end{align*}
where the equality follows from our assumption that $p_n/(\sqrt{n} \lambda_{nd}) \rightarrow 0$. The claim now follows using our earlier statement that $ \mathcal{O}_p(1) + \Omega_\infty(1) = \Omega_\infty(1)$.


\end{proof}

\begin{proof}[Proof of Corollary \ref{cor:dimension_estimation}]
We have,
\begin{align*}
    \mathbb{P}(\hat{d} = d) =& \mathbb{P} \left( \bigcap_{k = 0}^{d - 1} \{ g_{n, k} > c_n \} \cap \{ g_{n, d} \leq c_n \} \right) \\
    =& 1 - \mathbb{P} \left( \bigcup_{k = 0}^{d - 1} \{ g_{n, k} \leq c_n \} \cup \{ g_{n, d} > c_n \} \right) \\
    \geq& 1 - \sum_{k = 0}^{d - 1} \mathbb{P}(g_{n, k} \leq c_n) - \mathbb{P}(g_{n, d} > c_n)\\
    =& 1 - \sum_{k = 0}^{d - 1} \mathbb{P}\{ g_{n, k}/n \leq \mathcal{O}(1) \} + o(1)\\
    =& 1 + o(1),
\end{align*}
where $\mathbb{P}(g_{n, d} > c_n) = o(1)$ follows from $g_{n, d} = \mathcal{O}_p(1)$ (shown in Theorem \ref{theo:goes_to_zero}) and the final equality follows from $\mathbb{P}\{ g_{n, k}/n \leq \mathcal{O}(1) \} = o(1)$, $k = 0, \ldots, d - 1$, (shown in Theorem \ref{theo:power_for_small_k}).

\end{proof}

\bibliographystyle{apalike}
\bibliography{references}

\end{document}